\documentclass[12pt]{amsart}
\usepackage[utf8]{inputenc}
\usepackage{amssymb}
\usepackage{amsmath}
\usepackage{amsthm}
\usepackage{graphicx}
\usepackage{amscd}
\usepackage{epic}
\usepackage[alphabetic]{amsrefs}
\usepackage{xypic}
\usepackage[colorlinks, citecolor=blue]{hyperref}
\usepackage[left=1.5in,right=1.5in,top=1.5in,bottom=1.5in]{geometry}
\usepackage{soul}
\usepackage[pagewise]{lineno}
\usepackage{comment}
\usepackage{enumitem}
\usepackage{mathrsfs}
\usepackage{tikz-cd}
\usepackage{tikz}
\usepackage{float}

\theoremstyle{plain}
\newtheorem{theorem}{Theorem}[section]
\newtheorem{corollary}[theorem]{Corollary}
\newtheorem{lemma}[theorem]{Lemma}

\newtheorem{example}[theorem]{Example}

\newtheorem{proposition}[theorem]{Proposition}
\newtheorem{definition-lemma}[theorem]{Definition-Lemma}
\theoremstyle{remark}
\newtheorem{remark}[theorem]{Remark}

\theoremstyle{definition}

\def\>{\geq}
\newcommand{\mbQ}{\mathbb{Q}}
\newcommand{\mbR}{\mathbb{R}}

\newcommand{\Q}{{\mathbb{Q}}}

\newcommand{\R}{{\mathbb{R}}}

\newcommand{\mbC}{\mathbb{C}}

\DeclareMathOperator{\BirMori}{BirMori}
\DeclareMathOperator{\BirAut}{BirAut}

\def\mbF{\mathbb{F}}

\def\mbP{\mathbb{P}}
\def\>{\geq}

\def\dim{\operatorname{dim}}

\def\rank{\operatorname{rank}}

\theoremstyle{definition}
\newtheorem{definition}[theorem]{Definition}
\theoremstyle{definition}

\numberwithin{equation}{section}

\theoremstyle{remark}

\author{Roktim Mascharak}
\address{Physics Building, UCL, Gower St, London WC1E 6BT}
\email{roktim.mascharak.23@ucl.ac.uk}

\title[ Foliated Sarkisov Program on threefolds ]{ On the Log Sarkisov Program for foliations on projective threefolds}

\begin{document}
\maketitle
\begin{abstract}
     In this article, we prove the Sarkisov Program for co-rank one foliations with suitable singularities on normal projective threefolds. We also exibit a weaker version of birational super-rigidity between two foliated Mori fiber spaces with rank one foliations on normal projective threefolds.  
\end{abstract}
\tableofcontents

\section{Introduction}
In birational geometry, the MMP conjecture states that starting with a projective variety, the MMP terminates with either a minimal model (i.e., a variety with a nef canonical divisor) or a Mori fiber space (i.e., a fibration $f:X\rightarrow Z$ of relative Picard rank $1$ whose anti-canonical divisor is relatively ample). In general, these end products are not unique. The aim of the Sarkisov Program is to describe the rational map between two Mori fiber spaces (which are the end results of different MMPs starting from the same variety with non pseudoeffective
canonical divisor) in a more concrete manner; more precisely, to decompose this map into a composition of ``Sarkisov links''. \\

For three-dimensional normal quasi-projective varieties with a co-rank one foliation, the existence and termination of the MMP for an F-dlt foliated pair $(\mathcal{F},\Delta)$ are known due to \cite{CS21} and \cite{SS19}. Note that a $(K_{\mathcal{F}}+\Delta)$-negative Mori fiber space contraction is also a $K_X$-negative Mori fiber space by \cite[Lemma $8.12$]{spicer20}, where $X$ is a normal projective threefold and $(\mathcal{F},\Delta)$ is a co-rank one F-dlt pair. By \cite[Lemma $3.16$]{CS21}, $X$ has klt singularities if $(\mathcal{F},\Delta)$ has F-dlt singularities with $\lfloor \Delta\rfloor=0$. Hence, by \cite{HM09}, we can write the rational map between two Mori fiber spaces induced by two MMPs starting from the same variety in terms of ``classical Sarkisov links''. In this article, we show that this map can also be decomposed in terms of foliated Sarkisov links; this is our main theorem. Before stating the theorem let us begin to define \textbf{foliated Sarkisov links}.
Suppose that $\phi: (X,\mathcal{F}_X) \to S$ and $\psi:(Y,\mathcal{F}_Y) \to T$ are foliated Mori fibre spaces. A Sarkisov link $\sigma: X \dashrightarrow Y$ between $\phi$ and $\psi$ is one of four types:

\begin{center}
\begin{tabular}{cccc}
\textbf{I} & \textbf{II} & \textbf{III} & \textbf{IV} \\
\begin{tikzcd}
X' \arrow[r, dashed] \arrow[d] & Y \arrow[d, "\psi"] \\
X \arrow[d, "\phi"] & T\arrow[dl] \\
S &
\end{tikzcd}
&
\begin{tikzcd}
X' \arrow[r, dashed] \arrow[d] & Y' \arrow[d] \\
X \arrow[d, "\phi"] & Y \arrow[d, "\psi"] \\
S \arrow[r, equal] & T
\end{tikzcd}
&
\begin{tikzcd}
    X' \arrow[r, dashed] \arrow[d, "\phi"] & Y' \arrow[d] \\
    S \arrow[rd] & Y \arrow[d, "\psi"] \\
    & T
\end{tikzcd}
&
\begin{tikzcd}
X \arrow[rr, dashed] \arrow[d, "\phi"] & & Y \arrow[d, "\psi"] \\
S \arrow[rd, "s"] & & T \arrow["t", ld] \\
& R &
\end{tikzcd}
\\
\end{tabular}
\end{center}

There is a divisor $\Xi$ on the space $L$ on the top left (be it $L=X$ or $L=X'$) such that $(L,\mathcal{F}_L,\Xi)$ is F-dlt and $K_{\mathcal{F}_L}+\Xi$ is numerically trivial over the base (be it $S, T,$ or $R$). Every arrow which is not horizontal is an extremal contraction. If the target is $(X,\mathcal{F}_X)$ or $(Y,\mathcal{F}_Y)$, it is a divisorial contraction. The horizontal dotted arrows are compositions of $(K_{\mathcal{F}_L}+ \Xi)$-flops, $L_i \dashrightarrow L_{i+1}$, $1 \le i \le k$, where $L = L_1$. Links of \textbf{type IV} break into two types, IV$_m$ and IV$_s$. For a link of \textbf{type IV$_m$}, both $s:S\rightarrow R$ and $t:T\rightarrow R$ are foliated Mori fibre spaces. For a link of \textbf{type IV$_s$}, both $s$ and $t$ are small birational contractions. In this case $R$ is not $\mathbb{Q}$-factorial; for every other type of link, all varieties are $\mathbb{Q}$-factorial. Note that there is an induced birational map $\sigma: X \dashrightarrow Y$, but not necessarily a rational map between $S$ and $T$.

\begin{theorem}[cf. {\cite[Theorem $1.3$]{HM09}}]\label{MT}
Let $\phi:X\rightarrow S$ and $\psi: Y\rightarrow T$ be two foliated Mori fiber spaces which are the end products of two MMPs starting from $(Z,\mathcal{F},\Phi)$, where $Z$ is a normal projective $\mbQ$-factorial threefold and $(\mathcal{F},\Phi)$ is a co-rank one foliated F-dlt pair on $Z$ with $\lfloor \Phi\rfloor =0$. Then the induced rational map can be written as a composition of the four kinds of foliated Sarkisov links.

\end{theorem}

\begin{remark}
Unlike in the case of classical pairs, if the foliated pair $(\mathcal{F},\Delta)$ has F-dlt singularities and $\lfloor\Delta\rfloor=0$, this does not imply that the pair is foliated log terminal. For example, consider the foliation $\mathcal{F}$ on $\mathbb{A}^2$ given by the one-form $\omega=xdy+ydx$. Blowing up the origin, we obtain an invariant exceptional divisor with foliated discrepancy $a(E,\mathcal{F},0)=0=\epsilon(E)$. This implies that $(\mathcal{F},0)$ has a log canonical center at the origin; however, it is F-dlt.
\end{remark}

The main ingredients required to prove Theorem \ref{MT} are the finiteness of log terminal models for foliations and the log geography of log terminal models, which we develop in the third section of this article (see Theorem \ref{Finiteness} and Theorem \ref{LGLCM}). We adapt techniques from \cite{BCHM} and \cite{HM09}, with some modifications for the case of co-rank one foliations in our setup.\\ 

The classical Sarkisov program has been a very important tool for studying birational automorphism groups of rational varieties (see, e.g., \cite{ISK96}, \cite{BLZ19}, \cite{BFSZ24}, \cite{ED24}). Namely, any birational automorphism of a smooth rational variety (or, more generally, one with log-terminal singularities) can be decomposed into a composition of Sarkisov links. In the next section, we explore the relationship between the foliated Sarkisov program and the birational automorphism groups of foliated varieties by studying some examples.

In the final section, we address the case of a rank one foliation with canonical singularities on a normal projective threefold, where we prove the following theorem relating rank one foliated Mori fiber spaces:

\begin{theorem}\label{MT2}
Let $\phi:X\rightarrow S$ and $\psi:Y\rightarrow T$ be two foliated Mori fiber spaces which are the end products of two MMPs starting from $(Z,\mathcal{F},\Delta)$, where $X,Y$, and $Z$ are normal projective threefolds and $(\mathcal{F},\Delta)$ is a rank one foliated pair on $Z$ with canonical singularities (see Definition \ref{pairs}). Then there exist a non-empty common open set $V$ of $S$ and $T$ such that $\phi^{-1}(V)$ is isomorphic to $\psi^{-1}(V)$ as $V$-schemes. 
\end{theorem}

Theorem \ref{MT2} establishes a special \textbf{birational super rigidity} condition (in the sense of \cite[Section $3$]{DF23}) in the set-up of rank one foliated Mori fiber spaces.\\
The strategy we are using to prove Sarkisov program for co-rank one foliations does not work in the rank one case. This is because some of the core techniques employed in this strategy fail for rank one foliations. For example, Bertini-type theorems no longer hold in this setting. We can observe this in the following example.

\begin{example} 
Consider the fibration $\pi:\mbP^2\times\mbP^1\rightarrow \mbP^2$. Let $l$ be a line in $\mbP^2$ and let $C$ be a curve in $\mbP^2\times\mbP^1$ such that $\pi(C)=l$. Let $X$ be the blow-up of $\mbP^2\times \mbP^1$ along $C$ and let $\mathcal{F}$ be the rank one foliation induced by the projection $X\rightarrow \mbP^2$. Let $\tilde{D}$ be the strict transform of the divisor $\pi^{-1}(l)$ and $E$ be the exceptional divisor surjecting onto $C$. Then every point of $\tilde{C}=\tilde{D}\cap E$ is an F-lc center for the foliation $\mathcal{F}$. Since, for any ample Cartier effective divisor $H$, the intersection $H\cap \tilde{C}$ is non-empty, the pair $(\mathcal{F},H)$ can never be log canonical.
\end{example}

\textit{Acknowledgements.} The author would like to thank his advisors Professor Paolo Cascini and Professor Calum Spicer for suggesting the problem and for many useful discussions. The author would also like to thank Dr. Eduardo Alves da Silva, Professor Hamid Abban, Stefania Vassiliadis and Samuele Ciprietti for helpful advises and the anonymous referee for useful comments and suggestions which improved the quality of this article. This work was supported by the Engineering and Physical Sciences Research Council [EP/S021590/1] and The EPSRC Centre for Doctoral Training in Geometry and Number Theory (The London School of Geometry and Number Theory), University College London.
\section{Preliminaries}
In this article, every variety is defined over $\mbC$. We begin by recalling some basic notions of foliations on a normal variety. For more detailed discussions, we refer the reader to \cite{CS21} and \cite{CS20}.

\begin{definition}
    A foliation on a normal variety $X$ is a coherent subsheaf $\mathcal{F}\subset T_X$ such that:
    \begin{enumerate}
        \item $\mathcal{F}$ is saturated, i.e., $T_X/\mathcal{F}$ is torsion-free.
        \item $\mathcal{F}$ is closed under the Lie bracket.
    \end{enumerate}
    The rank of $\mathcal{F}$ is its rank as a sheaf, and the co-rank is $\dim(X)-\rank(\mathcal{F})$. We define the canonical divisor of the foliation, $K_{\mathcal{F}}$, to be a Weil divisor such that $\mathcal{O}_X(K_{\mathcal{F}})=(\det(\mathcal{F}))^*$. Given a foliation $\mathcal{F}$ of rank $r$ on a normal variety $X$, there is a natural map $\phi':\Omega^{[1]}_X\rightarrow \mathcal{F}^*$ where $\Omega^{[d]}_X=(\wedge^d \Omega_X)^{**}$, which is constructed from the inclusion $\mathcal{F}\subseteq T_X $. Using $\phi'$, we get the following map
    \[
    \phi: (\Omega^{[r]}_X\otimes \mathcal{O}_{X}(-K_{\mathcal{F}}))^{**}\rightarrow \mathcal{O}_X,
    \]
     The co-support of $\phi$ is defined to be the singular locus of the foliation.

    Let $S$ be a subvariety of a normal variety $X$, and let $\mathcal{F}$ be a rank-$r$ foliation on $X$. Then $S$ is said to be invariant under $\mathcal{F}$ if for any open set $U\subset X$ and any section $\partial\in H^0(U,\mathcal{F}_U)$, we have $\partial (I_{S\cap U})\subset I_{S\cap U}$, where $I_{S\cap U}$ is the ideal sheaf of $S\cap U$.
\end{definition}

\begin{definition}[Singularities of foliated pairs]\label{pairs}
    A foliated pair $(\mathcal{F},\Delta)$ on a normal variety $X$ consists of a foliation $\mathcal{F}$ on $X$ and an $\mbR$-divisor $\Delta$ such that $K_{\mathcal{F}}+\Delta$ is $\mbR$-Cartier. Given a birational morphism $\pi:X'\rightarrow X$ from a normal variety $X'$ and a foliated pair $(\mathcal{F},\Delta)$ on $X$, let $\mathcal{F}'$ be the pullback foliation on $X'$. We may write
    \[
    K_{\mathcal{F}'}+\pi_*^{-1}\Delta=\pi^*(K_{\mathcal{F}}+\Delta)+\sum a(E,\mathcal{F},\Delta)E,
    \]
    where the sum runs over all $\pi$-exceptional divisors $E$. We say that $(\mathcal{F},\Delta)$ is terminal (resp. canonical, log terminal, log canonical) if $a(E,\mathcal{F},\Delta)>0$ (resp. $\geqslant 0$, $>-\epsilon(E)$, $\geqslant -\epsilon(E)$) for any birational morphism $\pi:X'\rightarrow X$ and any $\pi$-exceptional divisor $E$ on $X'$. Here, $\epsilon(E)=0$ if $E$ is invariant under $\mathcal{F}'$ and $\epsilon(E)=1$ otherwise.

    Now consider $X$ to be a three-dimensional normal variety and $\mathcal{F}$ to be a co-rank one foliation on $X$. A pair $(\mathcal{F},\Delta)$ is called foliated divisorial log terminal (F-dlt) if:
    \begin{enumerate}
        \item each irreducible component of $\Delta$ is generically transverse to $\mathcal{F}$ and has a coefficient of at most one, and
        \item there exists a foliated log resolution (in the sense of \cite[Definition $3.1$]{CS21}) $\pi:Y\rightarrow X$ of $(\mathcal{F},\Delta)$ which extracts only divisors $E$ with discrepancy $>-\epsilon (E)$.
    \end{enumerate}
\end{definition}

\begin{remark}
    By \cite{CS20}, we know that F-dlt singularities for co-rank one foliations on normal threefolds are always non-dicritical.
\end{remark}

We introduce some classical definitions in the context of foliations which are analogous to those in \cite{BCHM}.

\begin{definition}
    Let $(\mathcal{F},\Delta)$ be a foliated pair on a normal variety $X$, and let $\phi:X\dashrightarrow Y $ be a proper birational contraction. We say that $\phi$ is $(K_{\mathcal{F}}+\Delta)$-negative if for some common resolution $p:W\rightarrow X$ and $q:W\rightarrow Y$, we can write
    \[
    p^*(K_{\mathcal{F}}+\Delta)=q^*(K_{\mathcal{F}'}+\Delta')+E,
    \]
    where $\mathcal{F}'$ is the transformed foliation and $\Delta'$ is the strict transform of $\Delta$ on $Y$, and $E\geqslant 0$ is a $q$-exceptional $\mbR$-divisor whose support contains the strict transform of every $\phi$-exceptional divisor.
\end{definition}

Next, we define the ample model and the log terminal model.
\begin{definition}
    Let $\pi:X\rightarrow U$ be a morphism of normal projective varieties and let $(\mathcal{F},\Delta)$ be a foliated pair with log canonical singularities. Let $\phi:X\dashrightarrow Y$ be a rational contraction of normal varieties over $U$.
    \begin{enumerate}
        \item If $\phi$ is birational, $(K_{\mathcal{F}}+\Delta)$-negative, and $\phi_*(K_{\mathcal{F}}+\Delta)$ is a nef divisor over $U$, we say that $\phi:X\dashrightarrow Y$ is a log terminal model of $(\mathcal{F},\Delta)$ over $U$.
        \item We say that $\phi:X\dashrightarrow Y$ is the ample model of a divisor $D$ on $X$ over $U$ if there is a relatively ample divisor $H$ on $Y$ over $U$ such that if $p:W\rightarrow X$ and $q:W\rightarrow Y$ is a common resolution of $\phi$, then $q$ is a projective contraction and we can write $p^*D\sim_{\mbR,U}q^* H+E$, where $E\geqslant 0$ is the fixed part of the linear system $|p^*D/U|_{\mbR}$.
    \end{enumerate}
\end{definition}

\begin{remark}
    Notice that in the definition of the log terminal model, we do not require the transformed pair on $Y$ to have F-dlt singularities.
\end{remark}

\begin{lemma}\label{CGT1}
    Let $\pi:X\rightarrow U$ be a projective morphism between $\mbQ$-factorial normal quasi-projective threefolds, and let $\mathcal{F}$ be a co-rank one foliation on $X$ with F-dlt singularities. Let $(\mathcal{F},\Delta=A+B)$ be an F-lc pair where $A$ is $\pi$-ample and $B$ is effective. Then, there exists a pair $\Delta'=A'+B'$ such that $(\mathcal{F},\Delta')$ is F-dlt, $K_{\mathcal{F}}+\Delta\sim_{\mbR,U}K_{\mathcal{F}}+\Delta'$, $A'\geq 0$ is $\pi$-ample, $B'\geqslant0$ is effective, and $\lfloor \Delta'\rfloor=0$. 
\end{lemma}

\begin{proof}
    Since $U$ is quasi-projective, we can find an ample divisor $H$ on $U$ such that for any $\pi$-ample divisor $G$, $G+\pi^*H$ is ample on $X$. Let $\delta>0$ be a sufficiently small rational number such that $A+\delta B$ is $\pi$-ample. We write $K_{\mathcal{F}}+A+B=K_{\mathcal{F}}+A+\delta B+(1-\delta)B$, and note that $(\mathcal{F},(1-\delta)B)$ is F-lc. Notice that $(\mathcal{F},(1-\delta)B)$ is not just F-lc but also F-dlt. Choose a sufficiently ample divisor $H$ on $U$ such that $A+\delta B+\pi^*H$ is an ample divisor on $X$. By \cite[Lemma $3.24$]{CS21}, we can choose $A'\sim_{\mbR}A+\delta B+\pi^*H$ such that $(\mathcal{F},A'+(1-\delta)B)$ is F-dlt. This gives the required pair.
\end{proof}

\begin{remark}\label{relBer}
    The proof of the above lemma readily indicates that it is easy to generalize \cite[Lemma $3.24$]{CS21}, namely the Bertini theorem, for co-rank one F-dlt pairs $(\mathcal{F},\Delta)$ on threefolds with $\lfloor \Delta\rfloor=0$.
\end{remark}

Now we define some important polytopes for the Sarkisov program and the MMP relation for co-rank one foliations on normal quasi-projective threefolds.

\begin{definition}\cite[Definition $1.1.4$]{BCHM}
    Let $\pi:X\rightarrow U$ be a morphism between normal quasi-projective varieties, and let $V$ be a finite-dimensional affine subspace of $\text{WDiv}_{\mbR}(X)$. Fix an $\mbR$-divisor $A\geqslant 0$. Let $\mathcal{F}$ be a foliation on $X$. Then we define the following sets:
    \begin{enumerate}
        \item $V_A=\lbrace \Delta \mid \Delta=A+B, B\in V\rbrace$.
        \item $\mathcal{L}_A(V)=\lbrace \Delta=A+B\in V_A \mid (\mathcal{F},\Delta) \text{ is F-lc and } B\geqslant 0\rbrace$.
        \item $\mathcal{E}_{A,\pi}(V)=\lbrace \Delta\in \mathcal{L}_A(V) \mid K_{\mathcal{F}}+\Delta\text{ is pseudo-effective over $U$ }\rbrace$.
        \item $\mathcal{N}_{A,\pi}(V)=\lbrace \Delta\in \mathcal{L}_{A}(V) \mid K_{\mathcal{F}}+\Delta \text{ is nef over $U$ }\rbrace$.
        \item Given a birational map $\phi:X\dashrightarrow Y$ over $U$, define 
        $\mathcal{L}_{\phi,A,\pi}(V)=\lbrace \Delta\in \mathcal{E}_{A,\pi}(V) \mid \phi \text{ is a log terminal model for } K_{\mathcal{F}}+\Delta \text{ over } U\rbrace$.
        \item Finally, given a rational map $\psi:X\dashrightarrow Z$ over $U$, we define
        $\mathcal{A}_{\psi,A,\pi}(V)=\lbrace \Delta\in \mathcal{E}_{A,\pi}(V) \mid \psi \text{ is the ample model for } K_{\mathcal{F}}+\Delta \text{ over } U\rbrace$.
    \end{enumerate}
\end{definition}

\begin{definition}
    Two foliated Mori fiber spaces $\phi:(X,\mathcal{F}_X)\rightarrow S$ and $\psi:(Y,\mathcal{F}_Y)\rightarrow T$ are said to be foliated log MMP-related if they are the end products of running a $(K_{\mathcal{F}}+\Phi)$-MMP, where $\mathcal{F}$ is a co-rank one foliation on a $\mbQ$-factorial normal projective threefold $Z$ and $({\mathcal{F}},\Phi)$ is an F-dlt pair with $\lfloor \Phi\rfloor =0$.
\end{definition}
\begin{remark}
 As $Z$ is $\mbQ$-factorial and by \cite[Section $6$]{CS21}, $\mbQ$-factoriality is preserved under Foliated MMP. Together with \cite[Lemma $3.16$]{CS21} this implies $X$ and $Y$ in the above definition have at worst $\mbQ$-factorial klt singularities.
\end{remark}

In this article, we assume $X,Y,Z$ in the above definition are all threefolds. Next, we need the idea of a Shokurov polytope in the foliated setup. 

\begin{theorem}\label{SPT}
    Given a ray $R\subset \overline{NE}(X)$, let
    \[
    R^{\bot}=\lbrace \Delta\in \mathcal{L}_{A}(V) \mid (K_{\mathcal{F}}+\Delta)\cdot R=0\rbrace.
    \]
    Let $\pi:X\rightarrow U$ be a projective morphism of normal quasi-projective varieties with $X$ having $\mbQ$-factorial singularities, and let $\mathcal{F}$ be a co-rank one foliation on $X$. Let $V$ be a finite-dimensional affine subspace of $WDiv_{\mbR}(X)$, defined over $\mbQ$, such that the support of all elements of $V$ is generically transverse to the foliation. Fix an ample $\mbQ$-divisor $A$ over $U$. Suppose that for some $\mbQ$-divisor $\Delta_0$ on $X$, the pair $(\mathcal{F},\Delta_0)$ is F-dlt.

    Then $\mathcal{N}_{A,\pi}(V)$ is intersection of $\mathcal{L}_{A}(V)$ with $R^{\bot}$ for finitely many extremal rays of $\overline{NE}(X)$, in particular it is a rational polytope.

    Also, let $\phi:X\dashrightarrow Y$ be any birational contraction over $U$. Then $\mathcal{L}_{\phi,A,\pi}(V)$ is a rational polytope. Moreover, there are finitely many morphisms $f_i:Y\rightarrow Z_i$ over $U$, $1\leqslant i\leqslant k$, such that if $f:Y\rightarrow Z$ is any contraction over $U$ and there is a divisor $D$ on $Z$ ample over $U$ such that $K_{\mathcal{F}_Y}+\Gamma=\phi_*(K_{\mathcal{F}}+\Delta)\sim_{\mbR,U}f^*D$ for some $\Delta\in\mathcal{L}_{\phi,A,\pi}(V)$, then there exist an index $i$ and an isomorphism $\eta: Z_i\rightarrow Z$ such that $f=\eta\circ f_i$.
\end{theorem}

\begin{proof}
    The proof is analogous to that of \cite[Theorem $3.11.1$ and Corollary $3.11.2$]{BCHM}, as the foliation $\mathcal{F}$ does not affect the structure of $\overline{NE}(X/U)$, and we have cone theorem for $K_{\mathcal{F}}$ due to \cite[Theorem $3.31$]{CS21}. We present an outline. Since $\mathcal{L}_{A}(V)$ is compact, it is sufficient to prove this locally around any point $\Delta\in \mathcal{L}_{A}(V)$. Choose $\Delta'$ sufficiently close to $\Delta$ such that $\Delta'-\Delta+A/2$ is ample over $U$. Let $R$ be an extremal ray over $U$ such that $(K_{\mathcal{F}}+\Delta')\cdot R=0$, where $\Delta'\in \mathcal{L}_{A,\pi}(V)$. We have
    \[
    (K_{\mathcal{F}}+\Delta-A/2)\cdot R=(K_{\mathcal{F}}+\Delta')\cdot R-(\Delta'-\Delta+A/2)\cdot R<0.
    \]
    Since $\Delta=A+B$ for some ample divisor $A$ and effective divisor $B$, by Lemma \ref{CGT1} we can assume $(\mathcal{F},\Delta)$ to be F-dlt by moving in the linear series $|K_{\mathcal{F}}+\Delta|_{\mbR,U}$. By \cite[Theorem 3.31]{CS21}, there are only finitely many such rays. $\mathcal{N}_{A,\pi}(V)$ is a closed subset of $\mathcal{L}_{A}(V)$. If $K_{\mathcal{F}}+\Delta$ is not nef over $U$, then by the cone theorem, $K_{\mathcal{F}}+\Delta$ is negative on a rational curve tangent to $\mathcal{F}$ which generates an extremal ray $R$ of $\overline{NE}(X)$. Thus $\mathcal{N}_{A,\pi}(V)$ is the intersection of $\mathcal{L}_{A}(V)$ with the half-spaces determined by finitely many extremal rays of $\overline{NE}(X/U)$. This proves that $\mathcal{N}_{A,\pi}(V)$ is a rational polytope.

    Similarly to \cite[Corollary $3.11.2$]{BCHM}, we can prove that $\mathcal{L}_{\phi,A,\pi}(V)$ is a rational polytope. For the finiteness of ample models, let $f:Y\rightarrow Z$ be a contraction over $U$ such that
    \[
    K_{\mathcal{F}_Y}+\Gamma=K_{\mathcal{F}_Y}+\phi_*\Delta\sim_{\mbR,U}f^*D,
    \]
    where $\Delta\in \mathcal{L}_{\phi,A,\pi}(V)$ and $D$ is an ample $\mbR$-divisor over $U$. Then $\Gamma$ belongs to the interior of a unique face $G=R^{\bot}$ of $\mathcal{N}_{C,\psi}(W)$, where $W$ is the image of $V$ under $\phi$, $C=\phi_*A$, and the curves in $R$ are all the curves contracted by $f$. Now $\Delta$ belongs to the interior of a unique face $F$ of $\mathcal{L}_{\phi,A,\pi}(V)$, and $G$ is determined by $F$. Since $\mathcal{L}_{\phi,A,\pi}(V)$ is a rational polytope, it has finitely many faces. By the rigidity lemma, $f$ is determined by $R$, which completes the proof.
\end{proof}

\begin{corollary}\cite[Corollary $3.11.3$]{BCHM}\label{SPT2}
    With the same setup as Theorem $\ref{SPT}$, let $(\mathcal{F},\Delta_0)$ be a dlt pair with $\lfloor \Delta_0\rfloor=0$. Let $f:X\rightarrow Z$ be a morphism over $U$ such that $\Delta_0\in \mathcal{L}_A(V)$ and $K_{\mathcal{F}}+\Delta_0\sim_{\mbR,U}f^* H$, where $H$ is an ample divisor over $U$. Let $\phi:X\dashrightarrow Y$ be a birational map over $Z$.

    Then there is a neighbourhood $P_0$ of $\Delta_0$ in $\mathcal{L}_A(V)$ such that for all $\Delta\in P_0$, if $\phi$ is a log terminal model for $K_{\mathcal{F}}+\Delta$ over $Z$, then $\phi$ is a log terminal model for $K_{\mathcal{F}}+\Delta$ over $U$.
\end{corollary}

\begin{proof}
    The proof is the same as in \cite[Corollary $3.11.3$]{BCHM}.
\end{proof}

Now, to construct the morphisms involved in the Sarkisov diagrams, we will need a basepoint-free theorem in a slightly more general setup than \cite{CS21}, which essentially follows from the proof of \cite[Theorem $9.4$]{CS21}.

\begin{proposition}\label{BPFT}
    Let $\pi:X\rightarrow Y$ be a projective morphism of normal quasi-projective varieties, with $X$ being potentially klt. Let $\mathcal{F}$ be a co-rank one foliation on $X$, and let $(\mathcal{F},\Delta)$ be an F-dlt pair. Suppose $\Delta=A+B$ where $A\geqslant 0$ is ample over $Y$ and $B$ is effective. If $K_{\mathcal{F}}+\Delta$ is nef over $Y$, then it is semi-ample over $Y$.  
\end{proposition}

\begin{proof}
    In \cite[Theorem $9.4$]{CS21}, this theorem is proved for the case when $Y$ is a point; we extend it to the relative setup. The same proof majorly works for this set-up. The main ingredients are the existence of relevant contractions as in \cite[Theorem $8.4$]{CS21} in the case of projective morphism of quasi-projective varieties and relative foliated canonical bundle formula. To extend \cite[Theorem $8.4$]{CS21} in this set-up, we need existence and termination of MMP for projective morphism between normal quasi-projective varieties (cf. proof of \cite[Theorem $8.3$]{CS21}), which exists by \cite[Corollary $2.3$]{SS19}. Rest of the proof for \cite[Theorem $8.4$]{CS21} in the relative set-up works verbatim. For canonical bundle formula, we replace \cite[Lemma $9.1$]{CS21} by \cite[Theorem $2.3.2$]{CHLX23}.

    For the convenience of the reader, we present an outline. The assumption that $X$ is potentially klt guarantees the existence of a small $\mbQ$-factorial modification $g:\overline{X}\rightarrow X$ over $Y$ such that if we write $g^*(K_{\mathcal{F}}+\Delta)=K_{\overline{\mathcal{F}}}+\overline{\Delta}$, then for any $\epsilon>0$ we can find a divisor $\Theta$ with $(1-\epsilon)\overline{\Delta}\leqslant\Theta\leqslant \overline{\Delta}$ such that $(\overline{\mathcal{F}},\Theta)$ is F-dlt (cf. \cite[Theorem $2.23$ and Lemma $3.29$]{CS21}). Let us call this \textbf{Property $\Theta$}.

    Let $\Gamma=\frac{1}{2}A+B$, and let $\lambda$ be the nef threshold of $K_{\mathcal{F}}+\Gamma$ with respect to $A$ over $Y$. If $\lambda<\frac{1}{2}$, we are done, as $K_{\mathcal{F}}+\Delta$ would be ample over $Y$. Suppose $\lambda=\frac{1}{2}$. Then we get a contraction $f:X\rightarrow X'$, which contracts every curve in a $(K_{\mathcal{F}}+\Gamma)$-negative and $(K_{\mathcal{F}}+\Delta)$-trivial extremal ray $R$ of $\overline{NE}(X/Y)$, which exists by \cite[Theorem $8.4$]{CS21}. We have $\rho(X'/Y)<\rho(X/Y)$. Let $H_R$ be the supporting Cartier divisor of $R$. If $H_R$ is big over $Y$, $f$ is either a divisorial or a flipping contraction. In either case, we can replace $X$ with $X'$, and $(\mathcal{F},\Delta)$ by $(\mathcal{F}',A'+B')$ with $A'\geqslant0$ ample over $Y$ and $B'\geqslant0$, and we continue. Note that, we construct such $A'$ and $B'$ with $K_{\mathcal{F}'}+A'+B'\sim_{\Q,U}K_{\mathcal{F}'}+f'_*(\Delta)$ as in the proof of \cite[Lemma $9.3$]{CS21}. It is easy to check that $X'$ also has \textbf{Property $\Theta$} by using \cite[Lemma $8.5$]{CS21} .

    If $H_R$ is not big over $Y$, then by \cite[Corollary $2.28$]{spicer20}, $X$ is covered by rational curves tangent to $\mathcal{F}$ spanning $R$. We run a $K_{\overline{X}}$-MMP $\psi:\overline{X}\dashrightarrow W$ which is $(K_{\overline{\mathcal{F}}}+\overline{\Delta})$-trivial over $Y$. Let $\psi_*\overline{\mathcal{F}}=\mathcal{H}$, and $\Gamma=\psi_*\overline{\Delta}$. It suffices to show that $K_{\mathcal{H}}+\Gamma$ is semi-ample over $Y$. This MMP terminates with a $K_W$-Mori fiber space over $Y$. In fact, fibers of this Mori fiber space are tangent to the transformed foliation as this is also $K_{\mathcal{H}}$-Mori fiber space. Using the foliated canonical bundle formula \cite[Theorem $2.3.2$]{CHLX23} we descend the foliation to the base of this Mori fiber space and as in \cite[Lemma $9.3$]{CS21}, we can reduce the problem to a lower-dimensional base and proceed by induction to obtain our claim.
\end{proof}

\begin{lemma}\label{BPF}
    Let $\pi:X\rightarrow U$ be a projective morphism from a normal quasi-projective threefold with $\mbQ$-factorial singularities, and let $\mathcal{F}_X$ be a co-rank one foliation. Let $\Delta$ be a divisor on $X$ such that $\Delta=f_*(A+B)$, where $f:Z\dashrightarrow X$ is a birational map over $U$, $A\geqslant 0$ is ample over $U$, and $B$ is effective. Let $\mathcal{F}=f^{-1}_*\mathcal{F}_X$ be the pullback foliation on $Z$. Suppose that $f$ is $(K_{\mathcal{F}}+A+B)$-negative, $(\mathcal{F},A+B)$ is an F-lc pair where $\mathcal{F}$ has F-dlt singularities, and $K_{\mathcal{F}_X}+\Delta$ is nef over $U$. Then $K_{\mathcal{F}_X}+\Delta$ is semi-ample over $U$.
\end{lemma}

\begin{proof}
    By Lemma \ref{CGT1}, we can assume that $(\mathcal{F},A+B)$ is an F-dlt pair with $\lfloor A+B\rfloor=0$. Let $H$ be a general ample $\mbQ$-divisor on $X$. After possibly replacing $H$ with a sufficiently small multiple, we may assume that if $H_Z$ is the strict transform of $H$ on $Z$, then $A-H_Z$ is ample over $U$. By Remark \ref{relBer}, for a sufficiently small $\epsilon>0$, there exists an effective divisor $C\sim_{\mbR,U} A-H_Z$ such that $(\mathcal{F},A+B+\epsilon C)$ is F-dlt and $f$ is still $(K_{\mathcal{F}}+A+B+\epsilon C)$-negative. Thus, by the negativity lemma and \cite[Lemma $3.11$]{CS21}, it follows that $(\mathcal{F}_X,\Delta+f_*(\epsilon C))$ is also F-dlt. Now $f_*A\sim_{\mbR,U}f_*C+H$. Thus, for a sufficiently small rational number $\delta>0$, we may choose $A'=\delta H$, $C'=(1-\delta)f_*A+\delta f_*C$ and denote $\Delta':=A'+C'+f_*B$. By construction, $f_*\Delta \sim_{\mbR,U} \Delta'$ and the resulting pair $(\mathcal{F}_X, \Delta')$ is F-dlt, with $A'$ being ample over $U$. Hence, by Proposition \ref{BPFT}, $K_{\mathcal{F}_X}+\Delta$ is semi-ample over $U$.
\end{proof}

\section{Log Geography of Log Terminal Models}
By \cite[Theorem $2.1$]{SS19}, we have the existence of a Mori fiber space for an F-dlt pair $(\mathcal{F},\Delta)$ where $\mathcal{F}$ is a co-rank one foliation on a $\mbQ$-factorial normal projective threefold $X$. By \cite[Theorem $1.2$]{CS21} and \cite[Theorem $2.6$]{SS19}, we know of the existence of log terminal models. We now proceed to prove the finiteness of log terminal models.

\begin{theorem}\label{Finiteness}
    Let $\pi:X\rightarrow U$ be a morphism of normal quasi-projective varieties, where $\dim X=3$, $X$ has $\mbQ$-factorial singularities, and $\mathcal{F}$ is a co-rank one foliation on $X$. Let $V$ be a finite-dimensional affine subspace of $WDiv_{\mbR}(X)$ defined over $\mbQ$. Fix a general ample $\mbQ$-divisor $A$ over $U$. Let $\mathcal{C}\subset \mathcal{L}_A(V)$ be a rational polytope. Suppose there exists a divisor $D_0\in V$ with $\lfloor D_0\rfloor=0$ such that $(\mathcal{F},D_0)$ is F-dlt.

    Then there are finitely many rational maps $\phi_i:X\dashrightarrow Y_i$ over $U$, $1\leqslant i\leqslant k$, with the property that if $\Delta\in \mathcal{C}\cap \mathcal{E}_{A,\pi}(V)$, then there exists an index $i$ such that $\phi_i$ is a log terminal model of $K_{\mathcal{F}}+\Delta$ over $U$.    
\end{theorem}

\begin{proof}
    Suppose $\Delta\in \mathcal{C}$. By Lemma \ref{CGT1}, we can choose $K_{\mathcal{F}}+\Delta'\sim_{\mbR,U}K_{\mathcal{F}}+\Delta$ such that $(\mathcal{F},\Delta')$ is an F-dlt pair. A map $\psi:X\dashrightarrow Z$ is a log terminal model of $K_{\mathcal{F}}+\Delta$ over $U$ if and only if it is a log terminal model of $K_{\mathcal{F}}+\Delta'$ over $U$, by application of the negativity lemma. Thus, we can assume that for any $\Delta\in \mathcal{C}$, the pair $(\mathcal{F},\Delta)$ is F-dlt. By replacing $V_A$ with the span of $\mathcal{C}$, we can also assume that $\mathcal{C}$ spans $V$.

    First, assume there is a divisor $\overline{\Delta}\in \mathcal{C}$ such that $K_{\mathcal{F}}+\overline{\Delta}\sim_{\mbR,U}0$. Pick any $\Theta\in\mathcal{C}$, $\Theta\neq \overline{\Delta}$. Then $\Theta=\lambda\Delta+(1-\lambda)\overline{\Delta}$ for some $0<\lambda\leqslant 1$ and some $\Delta$ on the boundary of $\mathcal{C}$. Now $K_{\mathcal{F}}+\Theta \sim_{\mbR,U}\lambda (K_{\mathcal{F}}+\Delta)$. In particular, $\Delta\in\mathcal{E}_{A,\pi}(V)$ if and only if $\Theta\in\mathcal{E}_{A,\pi}(V)$, and $(\mathcal{F},\Delta)$ and $(\mathcal{F},\Theta)$ have the same log terminal models over $U$. Since the boundary of $\mathcal{C}$ is contained in finitely many affine hyperplanes defined over $\mbQ$, we are done by induction on the dimension of $\mathcal{C}$.

    We now prove the general case. Since $\mathcal{L}_{A}(V)$ is compact and $\mathcal{C}\cap \mathcal{E}_{A,\pi}(V)$ is closed, it is sufficient to prove the result locally around any divisor $\Delta$ in the set. Since $\mathcal{F}$ has F-dlt singularities, we can find $\Theta_0\in \mathcal{L}_{A}(V)$ such that $(\mathcal{F},\Theta_0)$ is an F-dlt pair with $\lfloor\Theta_0\rfloor=0$. Let $\phi:X\dashrightarrow Y$ be a log terminal model of $(\mathcal{F},\Theta_0)$ over $U$.

Pick a neighbourhood $\mathcal{C}_0\subset \mathcal{L}_{A}(V)$ of $\Theta_0$, which is a rational polytope, such that for any $\Delta\in\mathcal{C}_0$, the map $\phi$ is $(K_{\mathcal{F}}+\Delta)$-nonpositive and $(\mathcal{F},\Delta)$ is an F-dlt pair with $\lfloor \Delta\rfloor =0$. Let $\phi_*\Theta_0=\Gamma_0$ and $\phi_*\Delta=\Gamma$. Since $(\mathcal{F}_Y,\Gamma_0)$ is an F-dlt pair and $Y$ is $\mbQ$-factorial, after possibly shrinking $\mathcal{C}_0$, we may assume that $(\mathcal{F}_Y,\Gamma)$ is an F-dlt pair for all $\Delta\in \mathcal{C}_0$. In particular, by replacing $\mathcal{C}$ with $\mathcal{C}_0$, we may assume that the rational polytope $\mathcal{C}'=\phi_*(\mathcal{C})$ is contained in $\mathcal{L}_{\phi_* A}(W)$, where $W=\phi_*(V)$ and the base change is $\pi':Y\rightarrow U$. 

By the proof of Lemma \ref{BPF}, there exist an affine linear isomorphism $L:W\rightarrow V'$ and a general ample $\mbQ$-divisor $A'$ over $U$. These satisfy the conditions that $L(\mathcal{C}')\subset \mathcal{L}_{A'}(V')$, that $L(\Gamma)\sim_{\R,U}\Gamma$ for all $\Gamma\in \mathcal{C}'$, and that $(\mathcal{F}_Y,\Gamma)$ is an F-dlt pair for any $\Gamma\in L(\mathcal{C}')$.

Note that $\dim V'\leqslant \dim V$. As $L(\Gamma)\sim_{\R,U}\Gamma$ and $\phi$ is $(K_{\mathcal{F}}+\Delta)$-negative, it follows that any log terminal model of $(\mathcal{F}_Y,L(\Gamma))$ over $U$ is also a log terminal model of $(\mathcal{F},\Delta)$ for any $\Delta\in \mathcal{C}$. By replacing $\mathcal{F}$ with $\mathcal{F}_Y$ and $\mathcal{C}$ with $L(\mathcal{C}')$, we can therefore assume that $K_{\mathcal{F}}+\Theta_0$ is nef over $U$.

By Lemma \ref{BPF}, $K_{\mathcal{F}}+\Theta_0$ is semi-ample over $U$; hence, it has an ample model $\psi:X\rightarrow Z$ over $U$. In particular, $K_{\mathcal{F}}+\Theta_0\sim_{\mbR,Z} 0$. From what we have already proved, there are finitely many birational maps $\phi_i:X\dashrightarrow Y_i$ over $Z$, for $1\leqslant i\leqslant k$, such that for any $\Delta\in \mathcal{C}\cap \mathcal{E}_{A,\pi}(V)$, there is an index $i$ for which $\phi_i$ is a log terminal model for $K_{\mathcal{F}}+\Delta$ over $Z$. Since there are only finitely many such indices, after possibly shrinking $\mathcal{C}$, Corollary \ref{SPT2} implies that for any $\Delta\in \mathcal{C}$, $\phi_i$ is a log terminal model for $K_{\mathcal{F}}+\Delta$ over $Z$ if and only if it is a log terminal model for $K_{\mathcal{F}}+\Delta$ over $U$. 

Suppose that $\Delta\in \mathcal{C}\cap \mathcal{E}_{A,\pi}(V)$. Then $\Delta\in \mathcal{C}\cap \mathcal{E}_{A,\psi}(V)$, and there is an index $i$ such that $\phi_i$ is a log terminal model for $K_{\mathcal{F}}+\Delta$ over $Z$. This implies that $\phi_i$ is a log terminal model for $K_{\mathcal{F}}+\Delta$ over $U$.
\end{proof}

Now we come to the main theorem of this section, which is the counterpart of \cite[Corollary $1.1.5$]{BCHM} in the foliated setting.

\begin{theorem}\label{LGLCM}
    Let $\pi:X\rightarrow U$ be a morphism of normal projective varieties, where $\mathcal{F}$ is a co-rank one foliation on $X$, $\dim X=3$, and $X$ is $\mbQ$-factorial. Let $V$ be a finite-dimensional affine subspace of $WDiv_{\mbR}(X)$ defined over $\mbQ$. Suppose there is a divisor $\Delta_0\in V$ such that $(\mathcal{F},\Delta_0)$ is an F-dlt pair with $\lfloor\Delta_0\rfloor=0$. Let $A$ be a general ample $\mbQ$-divisor over $U$ which has no components in common with any elements of $V$.
    \begin{enumerate}
        \item There are finitely many birational contractions $\phi_i:X\dashrightarrow Y_i$ over $U$, $1\leqslant i\leqslant p$, such that
        \[
        \mathcal{E}_{A,\pi}(V)=\cup_{i=1}^{p} \mathcal{L}_i,
        \]
        where each $\mathcal{L}_i=\mathcal{L}_{\phi_i,A,\pi}(V)$ is a rational polytope. Moreover, if $\phi:X\dashrightarrow Y$ is a log terminal model of $(\mathcal{F},\Delta)$ over $U$ for some $\Delta\in \mathcal{E}_{A,\pi}(V)$, then $\phi=\phi_i$ for some $1\leqslant i\leqslant p$.
        \item There are finitely many rational maps $\psi_j:X\dashrightarrow Z_j$ over $U$, $1\leqslant j\leqslant q$, which partition $\mathcal{E}_{A,\pi}(V)$ into subsets $\mathcal{A}_j=\mathcal{A}_{\psi_j,A,\pi}(V)$.
        \item For every $1\leqslant i\leqslant p$, there exists a $j$ with $1\leqslant j\leqslant q$ and a morphism $f_{i,j}:Y_i\rightarrow Z_j$ such that $\mathcal{L}_i\subset \overline{\mathcal{A}}_j$.
    \end{enumerate}
    In particular, $\mathcal{E}_{A,\pi}(V)$ and each $\overline{\mathcal{A}}_j$ are unions of finitely many rational polytopes.
\end{theorem}

\begin{proof}
    To prove $(1)$ and $(2)$, by Theorem \ref{Finiteness} and Theorem \ref{SPT}, and since ample models are unique by the rigidity lemma, it suffices to prove that if $\Delta\in \mathcal{E}_{A,\pi}(V)$, then $K_{\mathcal{F}}+\Delta$ has both a log terminal model and an ample model over $U$. By Lemma \ref{CGT1}, we may assume that $(\mathcal{F},\Delta)$ is an F-dlt pair, as log terminal and ample models are invariant under $\mbR$-linear equivalence. The existence of a log terminal model and an ample model over $U$ is guaranteed by \cite{CS21} and \cite{SS19}.
    
    Part $(3)$ follows from the proof of Theorem \ref{SPT}.
\end{proof}
\section{The Sarkisov Program}
We now prove that any birational map between two foliated Mori fiber spaces that are log MMP-related can be factored into a composition of Sarkisov links; this is the Sarkisov program. In this section, we follow the approach of \cite{HM09}.

First, let us establish some notation. Given a rational contraction $f:Z\dashrightarrow X$, an ample $\mathbb{Q}$-divisor $A$ on $Z$, and a finite-dimensional affine subspace $V$ of $WDiv_{\mbR}(Z)$ which is defined over $\mbQ$, we define
\[
\mathcal{A}_{A,f}(V)=\lbrace \Theta\in \mathcal{E}_{A}(V) \mid f \text{ is the ample model of } (\mathcal{F},\Theta)\rbrace,
\]
where $\mathcal{F}$ is a co-rank-one foliation on $Z$, $\dim Z=3$, and $Z$ has $\mbQ$-factorial singularities. Let $\mathcal{C}_{A,f}(V)$ denote the closure of $\mathcal{A}_{A,f}(V)$. We also assume that there is a $\mbQ$-divisor $D_0\in V$ with $\lfloor D_0\rfloor=0$ such that $(\mathcal{F},A+D_0)$ is an F-dlt pair and $K_{\mathcal{F}}+A+D_0$ is a big divisor.

\begin{theorem}[cf. {\cite[Theorem $3.3$]{HM09}}]\label{Lsm1}
    There are finitely many rational contractions $f_i:Z\dashrightarrow X_i$ for $1\leqslant i\leqslant m$ with the following properties:
    \begin{enumerate}
        \item The collection $\lbrace \mathcal{A}_i=\mathcal{A}_{A,f_i} \mid 1\leqslant i\leqslant m\rbrace$ is a partition of $\mathcal{E}_A(V)$. Each $\mathcal{A}_i$ is a finite union of the interiors of rational polytopes. If $f_i$ is birational, then $\mathcal{C}_i=\mathcal{C}_{A,f_i}$ is a rational polytope.
        \item If $i, j$ are indices such that $\mathcal{A}_j\cap \mathcal{C}_i\neq \varnothing$, then there is a contraction $f_{i,j}:X_i\rightarrow X_j$ and a factorization $f_j=f_{i,j}\circ f_i$.
    \end{enumerate}
    Now suppose, in addition, that $V$ spans the N\'eron-Severi group of $Z$.
    \begin{enumerate}
        \setcounter{enumi}{2} 
        \item For any $i \in \{1, \dots, m\}$, let $\mathcal{C}$ be a connected component of $\mathcal{C}_i$ that intersects the interior of $\mathcal{L}_A(V)$. The following are equivalent:
        \begin{enumerate}
            \item $\mathcal{C}$ spans $V$.
            \item If $\Theta\in \mathcal{A}_i\cap \mathcal{C}$, then $f_i$ is a log terminal model of $K_{\mathcal{F}}+\Theta$.
            \item $f_i$ is birational and $X_i$ is $\mbQ$-factorial.
        \end{enumerate}
        \item If $i, j \in \{1, \dots, m\}$ are indices such that $\mathcal{C}_i$ spans $V$, and if $\Theta$ is a general point of $\mathcal{A}_j\cap \mathcal{C}_i$ which is also in the interior of $\mathcal{L}_A(V)$, then $\mathcal{C}_i$ and $\overline{NE}(X_i/X_j)^*\times \mbR^k$ are locally isomorphic in a neighbourhood of $\Theta$ for some $k\geqslant 0$. Furthermore, the relative Picard number of $f_{i,j}:X_i\rightarrow X_j$ is equal to the difference in the dimensions of $\mathcal{C}_i$ and $\mathcal{C}_j\cap \mathcal{C}_i$.
    \end{enumerate}
\end{theorem}
\begin{proof}
    Part (1) follows from Theorem \ref{LGLCM}. We begin by proving (2).\\

    Pick $\Theta\in \mathcal{A}_j\cap\mathcal{C}_i$ and $\Theta'\in \mathcal{A}_i$, and for $t \in (0,1]$, let $\Theta_t = t\Theta'+(1-t)\Theta$. By the finiteness of log terminal models, we may find a positive constant $\delta>0$ and a birational contraction $f:Z\dashrightarrow X$ which is a log terminal model of $K_{\mathcal{F}}+\Theta_t$ for all $t\in (0,\delta]$. By replacing $\Theta'$ with $\Theta_{\delta}$, we may assume $\delta=1$. By construction, $\Theta_t$ lies in $\mathcal{E}_A(V)$; in particular, $\Theta_t=A+D_t$ for some effective $D_t\in V$. If we set $\Delta_t=f_*\Theta_t$, then $K_{\mathcal{F}_X}+\Delta_t$ is nef, $(\mathcal{F}_X,\Delta_t)$ is an F-dlt pair, and $f$ is $(K_{\mathcal{F}}+\Theta_t)$-non-positive for $t\in [0,1]$. Lemma \ref{BPF} implies that $K_{\mathcal{F}_X}+\Delta_t$ is semi-ample, and so for each $t$ there is an induced contraction $g_t:X\rightarrow X_t$. In particular, we get contractions $g_i:X\rightarrow X_i$ together with ample divisors $H_{1/2}$ and $H_1$ such that $K_{\mathcal{F}_X}+\Delta_{1/2}=g_i^*H_{1/2}$ and $K_{\mathcal{F}}+\Delta_1=g_i^*H_1$. If we set $H_t=(2t-1)H_1+2(1-t)H_{1/2}$, then $K_{\mathcal{F}_X}+\Delta_t=g_i^*H_t$ for all $t\in [0,1]$. As $K_{\mathcal{F}_X}+\Delta_0$ is semi-ample, it follows that $H_0$ is semi-ample, and the associated contraction $f_{i,j}:X_i\rightarrow X_j$ is the required morphism.\\

    Now we prove part (3). Suppose $V$ spans the N\'eron-Severi group of $Z$ and that $\mathcal{C}$ spans $V$. Pick $\Theta$ in the interior of $\mathcal{C}\cap \mathcal{A}_i$. Let $f:Z\dashrightarrow X$ be a log terminal model of $K_{\mathcal{F}}+\Theta$. By Theorem \ref{Finiteness}, $f=f_j$ for some $j$, and $\Theta\in \mathcal{C}_j$. But then $\mathcal{A}_i\cap \mathcal{A}_j\neq \varnothing$, which implies $i=j$. If $f_i$ is a log terminal model of $K_{\mathcal{F}}+\Theta$, then by definition, $f_i$ is birational and $X_i$ is $\mbQ$-factorial.

    Finally, suppose that $f_i$ is birational and $X_i$ is $\mbQ$-factorial. Fix $\Theta\in \mathcal{A}_i$. Pick any divisor $B\in V$ such that $-B$ is ample, $K_{\mathcal{F}_{X_i}}+f_{i*}(\Theta+B)$ is ample, and $\Theta+B\in\mathcal{L}_A(V)$. Then $f_i$ is $(K_{\mathcal{F}}+\Theta+B)$-negative, so $\Theta+B\in \mathcal{A}_i$. Since we can find a basis for the N\'eron-Severi group of $Z$ consisting of ample divisors, this implies that $\mathcal{C}$ spans $V$. This establishes (3).\\

    The proof of (4) is exactly the same as the proof of \cite[Theorem $3.3(4)$]{HM09}, replacing the use of \cite[Corollary $3.11.3$]{BCHM} with our Corollary \ref{SPT2}.\\
\end{proof}

From now on, we assume $V$ is a two-dimensional space that satisfies all the conditions of the previous theorem.

\begin{lemma}[cf. {\cite[Lemma $3.5$]{HM09}}]\label{LSM2}
    Let $f:Z\dashrightarrow X$ and $g:Z\dashrightarrow Y$ be two rational contractions such that $\mathcal{C}_{A,f}$ is two-dimensional and $\mathcal{O}=\mathcal{C}_{A,f}\cap\mathcal{C}_{A,g}$ is one-dimensional. Assume $\rho(X)\geqslant \rho(Y)$ and that $\mathcal{O}$ is not contained in the boundary of $\mathcal{L}_A(V)$. Let $\Theta$ be an interior point of $\mathcal{O}$ and let $\Delta=f_*\Theta$.

    Then there is a rational map $\pi:X\dashrightarrow Y$ that factors $g$ as $g=\pi\circ f$. Furthermore, one of the following holds:
    \begin{enumerate}
        \item $\rho(X)=\rho(Y)+1$ and $\pi$ is a $(K_{\mathcal{F}_X}+\Delta)$-trivial morphism. In this case:
        \begin{enumerate}
            \item If $\pi$ is birational, then $\mathcal{O}$ is not contained in the boundary of $\mathcal{E}_A(V)$. This subdivides into two cases:
            \begin{enumerate}
                \item $\pi$ is a divisorial contraction and $\mathcal{O}\neq \mathcal{C}_{A,g}$.
                \item $\pi$ is a small contraction and $\mathcal{O}=\mathcal{C}_{A,g}$.
            \end{enumerate}
            \item If $\pi$ is a Mori fiber space, then $\mathcal{O}=\mathcal{C}_{A,g}$ is contained in the boundary of $\mathcal{E}_A(V)$.
        \end{enumerate}
        \item $\rho(X)=\rho(Y)$. In this case, $\pi$ is a $(K_{\mathcal{F}_X}+\Delta)$-flop, and $\mathcal{O}\neq \mathcal{C}_{A,g}$ is not contained in the boundary of $\mathcal{E}_A(V)$.
    \end{enumerate}
\end{lemma}
\begin{proof}
    As $V$ and $\mathcal{C}_{A,f}$ are both two-dimensional, by Theorem \ref{Lsm1}(3), $f$ is birational and $X$ is $\mbQ$-factorial. Let $h:Z\dashrightarrow W$ be the ample model corresponding to $K_{\mathcal{F}}+\Theta$. Since $\Theta$ is not in the boundary of $\mathcal{L}_A(V)$, if $\Theta$ belongs to the boundary of $\mathcal{E}_A(V)$, then $K_{\mathcal{F}}+\Theta$ is not big and so $h$ is not birational. As $\mathcal{O} \subseteq \mathcal{C}_{A,f} \cap \mathcal{C}_{A,g}$, there are morphisms $p:X\rightarrow W$ and $q:Y\rightarrow W$ of relative Picard number at most one by Theorem \ref{Lsm1}(4). So there are two possibilities: $\rho(X)=\rho(Y)+1$ or $\rho(X)=\rho(Y)$.
\begin{itemize}

   \item  Suppose we are in the first case. Then $q$ is the identity map and $\pi=p:X\rightarrow Y$ is a contraction such that $g=\pi\circ f$. If $\pi$ is birational, then $h$ is also birational, so $\mathcal{O}$ is not contained in the boundary of $\mathcal{E}_A(V)$. If $\pi$ is divisorial, then $Y$ is $\mbQ$-factorial, so $\mathcal{O}\neq \mathcal{C}_{A,g}$ by Theorem \ref{Lsm1}(3). If $\pi$ is a small contraction, then $Y$ is not $\mbQ$-factorial, so $\mathcal{C}_{A,g}=\mathcal{O}$ must be one-dimensional. If $\pi$ is a Mori fiber space, then $\mathcal{O}$ is contained in the boundary of $\mathcal{E}_A(V)$ and $\mathcal{O}=\mathcal{C}_{A,g}$.\\

   \item  Now suppose we are in the second case, $\rho(X)=\rho(Y)$. Then we have $\rho(X/W)=\rho(Y/W)=1$. Again, by Theorem \ref{Lsm1}(4), $p$ and $q$ cannot be divisorial, as the one-dimensional set $\mathcal{O}$ cannot span the two-dimensional space $V$. Thus $W$ is not $\mbQ$-factorial. Hence $p$ and $q$ are small contractions and $\pi$ is a $(K_{\mathcal{F}_X}+\Delta)$-flop.\\
    \end{itemize}
\end{proof}

We now define the setup for the main result. Let $\Theta=A+B$ be a point on the boundary of $\mathcal{E}_A(V)$ but in the interior of $\mathcal{L}_A(V)$. Let $\mathcal{T}_1, \mathcal{T}_2, \dots, \mathcal{T}_k$ be the two-dimensional polytopes $\mathcal{C}_i$ which contain $\Theta$. After possibly reordering, we assume that the intersections $\mathcal{O}_0 = \mathcal{T}_1 \cap \partial\mathcal{E}_A(V)$, $\mathcal{O}_k = \mathcal{T}_k \cap \partial\mathcal{E}_A(V)$, and $\mathcal{O}_i=\mathcal{T}_i\cap \mathcal{T}_{i+1}$ (for $i=1, \dots, k-1$) are all one-dimensional. Let $f_i:Z\dashrightarrow X_i$ be the rational contractions associated to $\mathcal{T}_i$ and $g_i:Z\dashrightarrow S_i$ be the rational contractions associated to $\mathcal{O}_i$. Set $f=f_1:Z\dashrightarrow X$ and $g=f_k:Z\dashrightarrow Y$. Let $\phi:X\rightarrow S=S_0$ and $\psi:Y\rightarrow T=S_k$ be the induced morphisms (cf. Fig. $1$.).

\begin{theorem}[cf. {\cite[Theorem $3.7$]{HM09}}]\label{SL1}
    Suppose $\Phi$ is any divisor on $Z$ such that $(\mathcal{F},\Phi)$ is an F-dlt pair with $\lfloor \Phi\rfloor=0$ and $\Theta-\Phi$ is ample. Then $\phi$ and $\psi$ are two Mori fiber spaces that are outputs of a $(K_{\mathcal{F}}+\Phi)$-MMP. If $\Theta$ is contained in at least two polytopes $\mathcal{T}_i$ ($k \ge 2$), they are connected by a Sarkisov link.
\end{theorem}
\begin{proof}
    The incidence relations between the polytopes yield birational maps as illustrated in the diagram below. For instance, since $X=X_1$ and $X'=X_2$ are both log terminal models for any $\Delta \in \mathcal{O}_1$, there is a map $p:X' \dashrightarrow X$. Similarly, there is a map $q:Y' \dashrightarrow Y$.
    \begin{figure}[h!]
        \centering
        \begin{tikzcd}
            X' \arrow[rr, dashed] \arrow[d, "p"', dashed] & & Y' \arrow[d, "q", dashed] \\
            X \arrow[d, "\phi"'] & & Y \arrow[d, "\psi"] \\
            S \arrow[dr, "s"'] & & T \arrow[dl, "t"'] \\
            & R &
        \end{tikzcd}
        \caption{Incidence relations induced by the polytopes}
        \label{fig:diagram}
    \end{figure}

    Let $R$ be the target of the ample model for $K_{\mathcal{F}}+\Theta$. Since $Z\dashrightarrow S_0$ and $Z\dashrightarrow S_k$ are rational contraction associated to $\mathcal{O}_0$ and $\mathcal{O}_k$ respectively, and both $\mathcal{O}_0$ and $\mathcal{O}_k$ are contained in the $\partial \mathcal{E}_A(V)$, by Lemma \ref{LSM2}, $\phi:X\rightarrow S$ and $\psi:Y\rightarrow T$ are Mori fiber spaces.
    
  Pick points $\Theta_1$ and $\Theta_k$ in the interiors of $\mathcal{T}_1$ and $\mathcal{T}_k$ respectively, both sufficiently close to $\Theta$, such that the divisors $\Theta_1-\Phi$ and $\Theta_k-\Phi$ are ample. Since $\lfloor\Phi\rfloor=0$, by \cite[Lemma $3.24$]{CS21} we can find an ample divisor $H$ such that $(\mathcal{F},\Phi+H)$ is an F-dlt pair and $K_{\mathcal{F}}+\Phi+H$ is ample. Furthermore, we can find a positive real number $\mu<1$ such that $\mu H\sim_{\mbR}\Theta_1-\Phi$. Note that $f_1=f$ is the ample model of $K_{\mathcal{F}}+\Theta_1+\mu H$. We can then pick a scalar $\lambda<\mu$ that is sufficiently close to $\mu$ such that $f$ is $(K_{\mathcal{F}}+\Phi+\lambda H)$-negative and remains the ample model for $K_{\mathcal{F}}+\Phi+\lambda H$.

It follows that $f$ is the unique log terminal model of $K_{\mathcal{F}}+\Phi+\lambda H$. This shows that if we run a $(K_{\mathcal{F}}+\Phi)$-MMP with scaling of $H$, the process yields the rational map $f$ when the scaling factor is $\lambda$. Since $X$ and $Y$ are both $\mbQ$-factorial varieties, it follows that the Mori fiber spaces defined by $\phi$ and $\psi$ are end products of a $(K_{\mathcal{F}}+\Phi)$-MMP. Let $\Delta=f_*\Theta$. Then the divisor $K_{\mathcal{F}_X}+\Delta$ is numerically trivial over $R$.
    
    By Theorem \ref{Lsm1}(4), for the contractions $X_i\rightarrow R$, we have $\rho(X_i/R)\leqslant 2$. If $\rho(X_i/R)=1$ and $X_i\rightarrow R$ is a Mori fiber space, then by Lemma \ref{LSM2} a facet of $\mathcal{T}_i$ is contained in the boundary of $\mathcal{E}_A(V)$, so $i=1$ or $k$. Thus $X_i\dashrightarrow X_{i+1}$ is a flop for $1 < i < k-1$, and $\rho(X_i/R)=2$ for $2\leqslant i\leqslant k-1$.
    Let $s:S\rightarrow R$ and $t:T\rightarrow R$ be the induced contractions. Lemma \ref{LSM2}(1) implies a dichotomy for $p$: either $p$ is a divisorial contraction and $s$ is the identity (if $\rho(X/R)=1$), or $p$ is a flop and $s$ is not the identity (if $\rho(X/R)=2$). A similar dichotomy holds for $q$ and $t$. This gives four cases:
    \begin{enumerate}
        \item If $s$ and $t$ are the identity, then $p$ and $q$ are divisorial contractions. This is a Sarkisov \textbf{link of type II}.
        \item If $s$ is the identity and $t$ is not, then $p$ is a divisorial contraction and $q$ is a flop. This is a \textbf{link of type I}.
        \item If $t$ is the identity and $s$ is not, we have a \textbf{link of type III}.
        \item Finally, suppose that neither $s$ nor $t$ is the identity. In this case, both $p$ and $q$ are flops.

    Let us assume, to seek a contradiction, that $s$ is a divisorial contraction. Let $F$ be the divisor contracted by $s$, and let $E$ be its inverse image in $X$. Since $\rho(X/S)=1$, we have $\phi^*(F)=mE$ for some natural number $m$. Using the same argument as in the proof of Lemma \ref{BPF}, after possibly replacing $\Delta$, we can ensure that $(\mathcal{F}_X,\Delta+\delta E)$ is an F-dlt pair for a sufficiently small $\delta>0$. As $K_{\mathcal{F}_X}+\Delta$ is semi-ample, the stable base locus is $E = \bold{B}(K_{\mathcal{F}_X}+\Delta+\delta E/R)$. If we run the $(K_{\mathcal{F}_X}+\Delta+\delta E)$-MMP over $R$, we obtain a birational contraction $X\dashrightarrow W$ which is a Mori fiber space over $R$. Since $\rho(X/R)=2$, we must have $W=Y$, which gives a \textbf{link of type III}. This is a contradiction.

    Therefore, $s$ cannot be a divisorial contraction. Similarly, $t$ can never be a divisorial contraction.

    This leaves two possibilities for the maps $s$ and $t$:
    \begin{itemize}
        \item If $s$ is a Mori fiber space, then $R$ is $\mbQ$-factorial, and so $t$ must be a Mori fiber space as well. This scenario corresponds to a link of \textbf{type IV$_m$}.
        \item If $s$ is a small contraction, then so is $t$. This gives a link of \textbf{type IV$_s$}.
    \end{itemize}.
    \end{enumerate}
\end{proof}
In the next section, we finally prove that we can set up such a two-dimensional space $V$ for which all the desired properties hold.
\section{Proof of the Main Theorem}
In this section, we prove Theorem \ref{MT}.

\begin{lemma}[cf. {\cite[Lemma $4.1$]{HM09}}]\label{FL}
    Let $\phi:X\rightarrow S$ and $\psi:Y\rightarrow T$ be two Sarkisov-related Mori fiber spaces corresponding to two $\mbQ$-factorial F-dlt pairs $(\mathcal{F}_X,\Delta)$ and $(\mathcal{F}_Y,\Gamma)$. Then we can find a smooth projective variety $Z$ with a co-rank-one foliation $\mathcal{F}$, two birational contractions $f:Z\dashrightarrow X$ and $g:Z \dashrightarrow Y$, an F-dlt pair $(\mathcal{F},\Phi)$, an ample divisor $A$ on $Z$, and a two-dimensional rational affine subspace $V$ of $WDiv_{\mbR}(Z)$ such that:
    \begin{enumerate}
        \item if $\Theta\in \mathcal{L}_A(V)$, then $\Theta-\Phi$ is ample;
        \item the sets $\mathcal{A}_{A,\phi\circ f}(V)$ and $\mathcal{A}_{A,\psi\circ g}(V)$ are not contained in the boundary of $\mathcal{L}_A(V)$;
        \item the closures $\mathcal{C}_{A,f}(V)$ and $\mathcal{C}_{A,g}(V)$ are two-dimensional;
        \item the closures $\mathcal{C}_{A,\phi\circ f}(V)$ and $\mathcal{C}_{A,\psi\circ g}(V)$ are one-dimensional; and
        \item $V$ satisfies the conditions of Theorem \ref{Lsm1}.
    \end{enumerate}
\end{lemma}

\begin{proof}
    As $\phi$ and $\psi$ are Sarkisov-related, by definition, we can find a $\mbQ$-factorial threefold $Z$ with a co-rank-one foliation $\mathcal{F}$ and an F-dlt pair $(\mathcal{F},\Phi)$ with $\lfloor \Phi\rfloor=0$, such that there exist birational maps $f:Z\dashrightarrow X$ and $g:Z\dashrightarrow Y$ which are outcomes of a $(K_{\mathcal{F}}+\Phi)$-MMP.

    Let $p:W\rightarrow Z$ be a foliated log resolution of $(\mathcal{F},\Phi)$ that resolves the indeterminacies of $f$ and $g$ and satisfies $a(E,\mathcal{F},\Phi)>-\epsilon(E)$ for all $p$-exceptional divisors $E$. Let $\mathcal{G}$ be the pullback foliation on $W$. We may write
    \[
    K_{\mathcal{G}}+p_*^{-1}\Phi+\sum \delta_i\epsilon(E_i)E_i=p^*(K_{\mathcal{F}}+\Phi)+E'+\sum \delta_i\epsilon(E_i)E_i,
    \]
    where the $E_i$ are all the $p$-exceptional divisors and the $\delta_i>0$ are sufficiently small rational numbers such that $E'+\sum\delta_i\epsilon(E_i)E_i$ is effective. If we run a $(\mathcal{G},p_*^{-1}\Phi+\sum\delta_i\epsilon(E_i)E_i)$-MMP over $Z$, it terminates by contracting the exceptional locus of $p$, and we recover the pair $(\mathcal{F},\Phi)$. By replacing $(Z, \Phi)$ with this new pair on $W$, we can assume from the start that $Z$ is smooth, $(\mathcal{F},\Phi)$ is log smooth, and $f$ and $g$ are morphisms.

    Let $A, H_1, \dots, H_k$ be ample $\mbQ$-divisors forming a basis for the Néron-Severi group of $Z$. Let $H = A + \sum_{i=1}^k H_i$. Pick sufficiently ample divisors $C$ on $S$ and $D$ on $T$. Then, for a small enough rational number $0 < \delta < 1$, the divisors $-(K_{\mathcal{F}_X}+\Delta+\delta f_*H)+\phi^*C$ and $-(K_{\mathcal{F}_Y}+\Delta+\delta g_*H)+\psi^*D$ are both ample, and $K_{\mathcal{F}}+\Phi+\delta H$ is negative with respect to both $f$ and $g$. By replacing $H$ with $\delta H$, we may assume $\delta=1$. Now, we can find a $\mbQ$-divisor $\Phi_0\geqslant \Phi$ such that $A+(\Phi_0-\Phi)$, $-(K_{\mathcal{F}_X}+f_*(\Phi_0+H))+\phi^*C$, and $-(K_{\mathcal{F}_Y}+g_*(\Phi_0+H))+\psi^*D$ are all ample.

    Next, choose general ample $\mbQ$-divisors $F_1\geqslant 0$ and $G_1\geqslant 0$ such that $F_1\sim_{\mbQ}-(K_{\mathcal{F}_X}+f_*(\Phi_0+H))+\phi^*C$ and $G_1\sim_{\mbQ}-(K_{\mathcal{F}_Y}+g_*(\Phi_0+H))+\psi^*D$. Let $F=f^*F_1$ and $G=g^*G_1$. By choosing $F_1$ and $G_1$ appropriately, as in \cite[Lemma $3.24$]{CS21}, we can ensure that $(\mathcal{F},\Phi_0+H+F+G)$ is an F-dlt pair.

    Let $V_0$ be the affine subspace of $WDiv_{\mbR}(Z)$ defined as the translate of $\text{span}_{\mbR}\{H_1, \dots, H_k, F, G\}$ by $\Phi_0$. For any $\Theta=A+B\in \mathcal{L}_A(V_0)$, the divisor $\Theta-\Phi=(A+\Phi_0-\Phi)+(B-\Phi_0)$ is ample, because $B-\Phi_0$ is a positive linear combination of ample divisors and is therefore nef, while $A+\Phi_0-\Phi$ is ample by construction. Note that $\Phi_0+G+H\in \mathcal{A}_{A,\psi\circ g}(V_0)$, and $f$ and $g$ are weak log canonical models for $K_{\mathcal{F}}+\Phi_0+F+H$ and $K_{\mathcal{F}}+\Phi_0+G+H$ respectively. Thus, $V_0$ satisfies the conditions of Theorem \ref{Lsm1}.

    Since $\{H_i\}$ generates the Néron-Severi group, we can find constants $\lbrace h_i\rbrace$ such that $G \equiv \sum h_i H_i$. For a small enough $\delta>0$, the divisor $\Phi_0+F+\delta G+H-\delta(\sum h_i H_i)$ is in $\mathcal{L}_A(V_0)$ and is numerically equivalent to $\Phi_0+F+H$. This shows that $\mathcal{A}_{A,\phi\circ f}(V_0)$ is not contained in the boundary of $\mathcal{L}_A(V_0)$. A similar argument holds for $\mathcal{A}_{A,\psi\circ g}(V_0)$. It follows that $\mathcal{C}_{A,f}(V_0)$ and $\mathcal{C}_{A,g}(V_0)$ span $V_0$, while $\mathcal{C}_{A,\phi\circ f}(V_0)$ and $\mathcal{C}_{A,\psi\circ g}(V_0)$ span affine hyperplanes, since $\rho(X/S)=\rho(Y/T)=1$.

    Finally, let $V_1$ be the two-dimensional affine space spanned by $\Phi_0+F+H-A$ and $\Phi_0+G+H-A$, translated by $A$. Let $V$ be a small general perturbation of $V_1$ that is defined over $\mbQ$. This $V$ is a two-dimensional subspace of $V_0$ that inherits the required properties, thus satisfying all five conditions of the lemma.
\end{proof}

Finally, we combine these results to prove our main theorem.

\begin{proof}[Proof of Theorem \ref{MT}]
    Let $\phi:X \rightarrow S$ and $\psi:Y \rightarrow T$ be two foliated Mori fiber spaces that are log MMP-related. By Lemma \ref{FL}, we can construct a smooth projective threefold $Z$ with a co-rank-one foliation $\mathcal{F}$, an F-dlt pair $(\mathcal{F},\Phi)$, an ample divisor $A$, and a two-dimensional rational affine subspace $V$ satisfying the conditions of the lemma.

    Let $\Theta_0$ be a point in the interior of $\mathcal{C}_{A,\phi\circ f}(V)$ and let $\Theta_1$ be a point in the interior of $\mathcal{C}_{A,\psi\circ g}(V)$, such that both points are in the interior of $\mathcal{L}_A(V)$. Since $\mathcal{E}_A(V)$ is a connected subset of the two-dimensional rational polytope $\mathcal{L}_A(V)$, its boundary is a union of one-dimensional segments.
    
    The boundary of $\mathcal{E}_A(V)$ must contain a path connecting the regions around $\Theta_0$ and $\Theta_1$. We can choose a path that lies entirely in the interior of $\mathcal{L}_A(V)$. This path corresponds to divisors $K_{\mathcal{F}}+\Theta$ that are pseudo-effective but not big. As we traverse this path, we move from one two-dimensional polytope $\mathcal{C}_{A,f_i}(V)$ to another. Each time we cross from one such polytope to the next, we pass through a common one-dimensional face. Let the sequence of these intersections be defined by points $\Theta_i$ as in the setup for Theorem \ref{SL1}.
    
    By Theorem \ref{SL1}, each transition from a model $X_i$ to $X_{i+1}$ corresponds to a Sarkisov link. Chaining these links together provides the required factorization of the rational map between $X$ and $Y$. Therefore, the two foliated Mori fiber spaces are connected by a sequence of foliated Sarkisov links.
\end{proof}

\section{Examples}
In this section, we explore some examples of the foliated Sarkisov program to provide evidence that it is more rigid than the classical Sarkisov program and to understand its relationship with the birational automorphism groups of foliations.

We begin by defining the birational automorphism group of a foliation.
\begin{definition}
    Let $X$ be a normal projective variety, let $\mathcal{F}$ be a foliation on $X$, and let $\phi:X\dashrightarrow X$ be a birational automorphism of $X$. Then $\phi$ is said to be a birational automorphism of $\mathcal{F}$ if there exists a Zariski-open set $U\subset X$ on which $\phi$ is an isomorphism and $\phi_*(\mathcal{F}|_U) = \mathcal{F}|_U$. We denote this group by $\BirAut_{\mathcal{F}}(X)$.
\end{definition}

\begin{remark}
    Note that by definition, $\BirAut_{\mathcal{F}}(X)$ is a subgroup of $\BirAut(X)$.
\end{remark}

\begin{example}
    Consider the foliation on $\mbP^1\times \mbP^1$ induced by the projection $f:\mbP^1\times \mbP^1\rightarrow \mbP^1$. The blow-up of a point $x \in \mbP^1\times \mbP^1$, which we denote by $W$, is isomorphic to the blow-up of $\mbP^2$ at two distinct points. This gives the following Sarkisov link:
    \begin{equation}
        \begin{tikzcd}
            & W \arrow[ld,"p"] \arrow[rd,"q"] & \\
            \mbP^1\times \mbP^1 \arrow[d,"f"] & & \mbP^2 \arrow[d,"g"] \\
            \mbP^1 & & \text{pt}
        \end{tikzcd}
    \end{equation}
    If we consider $\mbP^1\times\mbP^1\rightarrow \mbP^1$ and $\mbP^2\rightarrow \text{pt}$ as classical Mori fiber spaces, they are MMP-related. However, the map $q:W \rightarrow \mbP^2$ contracts strict transforms of the two rulings of $\mbP^1\times\mbP^1$ passing through $x$. One of them is transverse to the foliation induced by $f:\mbP^1\times\mbP^1\rightarrow \mbP^1$, so, this divisor is also transverse to the pullback foliation on $W$. Note that, by \cite[Theorem $3.31$]{CS21}, foliated MMP only contracts curves tangent to the foliation. Hence, these two Mori fiber spaces are not foliated MMP-related.\\
    In particular $q$ is composition of two classical Sarkisov links, connecting $W\rightarrow \mathbb{F}_1$ and $\mathbb{F}_1\rightarrow \mbP^2$. However, $W\rightarrow \mbP^2$ contracts a transverse divisor, which is not permitted in the foliated Sarkisov program. So, The only possible foliated Sarkisov links starting from $(\mbP^1\times\mbP^1, f)$ are sequences of links of Type 2, ending with a scroll $\mbF_n\rightarrow \mbP^1$ for some $n\geq 0$.
\end{example}

With this example in hand and inspired by \cite{BLZ19}, we define a foliated version of $\BirMori(X)$.
\begin{definition}
    Let $X$ be a normal projective variety and $\mathcal{F}$ be a co-rank one foliation on $X$. Let $X\rightarrow Y$ be a foliated Mori fiber space. We define $\BirMori_{\mathcal{F}}(X/Y)$ to be the groupoid of all birational maps between foliated Mori fiber spaces that can be decomposed into a sequence of foliated Sarkisov links (as in Theorem \ref{MT}).
\end{definition}

In the classical Sarkisov program, if two Mori fiber spaces with $\mbQ$-factorial terminal singularities are birational to each other, they are automatically the outputs of two different MMPs starting from the same variety. However, in the case of foliations, if $(X_1/Y_1,\mathcal{F}_1)$ and $(X_2/Y_2,\mathcal{F}_2)$ are two foliated Mori fiber spaces with $X_1$ and $X_2$ smooth, it is not clear whether they are end products of two MMPs starting from a common foliated variety $(W,\mathcal{F}_W)$. The issue is, we don't know the existence of a common foliated log resolution that only extracts exceptional divisors over foliated klt centers. Thus, in the definition of $\BirMori_{\mathcal{F}}(X)$, we must impose the condition that the foliated Mori fiber spaces are connected by foliated Sarkisov links.

Using the argument in the above example, it is easy to see that the groupoid associated with the foliated Mori fiber space $\mbP^1\times\mbP^1\rightarrow \mbP^1$ contains all Mori fiber spaces $\mathbb{F}_n\rightarrow \mbP^1$ for $n\geq 1$, but not $\mbP^2\rightarrow \text{pt}$. Hence, $\BirMori_{\mathcal{F}}(\mbP^1\times \mbP^1)$ is a strict sub-groupoid of $\BirMori(\mbP^1\times\mbP^1)$. In fact, as described in \cite[Remark $3.13$]{BL09}, this is precisely the de Jonqui\`eres group, which is isomorphic to $\mbP GL(2,\mbC(x))\rtimes \mbP GL(2,\mbC)$.\\

Next, we show some limitations of the foliated Sarkisov program for F-dlt pairs. Before going into the example we recall that any fano foliation on $\mbP^2$ is isomorphic to the radial foliation on $\mbP^2$, i.e. it is given by the rational map $[x_0:x_1:x_2]\dashrightarrow [x_1:x_2]$ to $\mbP^1$ (cf. \cite[Chapter $2$, Section $3$, Example $(2)$]{Bru}).
\begin{example}
    Consider a foliation $\mathcal{F}$ on $\mbP^2$ such that $K_{\mathcal{F}}$ is not pseudo-effective. Then $-K_{\mathcal{F}}$ must be ample, which implies that $\mathcal{F}$ is a Fano foliation. Such a foliation is isomorphic to the radial foliation on $\mbP^2$, and has strictly F-lc singularities; in particular, it has a dicritical singularity. It is therefore not possible to study birational automorphism groups of such foliations using the F-dlt Sarkisov program. To understand the issue, let us work with the only Fano foliation on $\mbP^2$ which is induced by the rational map $[x_0:x_1:x_2]\dashrightarrow [x_1:x_2]$ to $\mbP^1$. Suppose there exists a foliated Sarkisov link to a surface $S$ with a foliated Mori fiber structure $(S,\mathcal{F})\rightarrow T$, where $\mathcal{F}$ has non-dicritical singularities. $T$ cannot be a point as $\mathcal{F}$ has non-dicritical singularities, so $\dim T=1$. Hence, the only possibility is a link of type $1$. Which implies there exists a divisorial contraction $\pi:S\rightarrow \mbP^2$, which contracts only $\mathcal{F}$-invariant curves. However, since $\mathcal{F}$ has non-dicritical singularities and the starting Fano foliation on $\mbP^2$ has dicritical singularity, $\pi$ must contract an exceptional curve which is not $\mathcal{F}$-invariant, which is a contradiction.
   
\end{example}

This example shows that the F-dlt Sarkisov program is not sufficient to study birational automorphisms of the only Fano foliations on $\mbP^2$. To study them in a similar fashion, we would need to allow at least the contraction of $K_{\mathcal{F}}$-trivial curves and blow-ups along strictly F-lc centers.

In dimension three, the situation is more complicated as we do not have a full classification of Sarkisov links starting from $\mbP^3$. Noting that any Fano foliation on $\mbP^3$ must have at least a strictly F-lc singularity, we can easily discard some of the known links. We illustrate this for some of the links described in \cite{BL11}.

\begin{example}
    Consider the blow-up of a line $C$ in $\mbP^3$. This gives a classical link of type 1 (cf. Theorem \ref{MT}) with the target $Y$ being a $\mbP^2$-bundle over $T=\mbP^1$. The variety $Y$ is naturally endowed with a smooth foliation $\mathcal{F}_1$ induced by the projection $Y\rightarrow T$, which makes $Y\rightarrow T$ a foliated Mori fiber space. If this is a foliated link, the transformed foliation $\mathcal{F}_2$ on $\mbP^3$ would have to be a Fano foliation as the Picard rank of $\mbP^3$ is one. Note that, since the foliation defined by $Y\rightarrow T$ is algebraically integrable, by the construction of foliated Sarkisov program $\mathcal{F}_2$ is also algebraically integrable (cf. \cite{ACSS}). $\mathcal{F}_2$ is not defined by a morphism as it is a foliation on $\mbP^2$, so this is defined by a dominant rational map $\phi:\mbP^3\dashrightarrow C$. If we consider the closure of the general fibers of $\phi$, all of them intersect along the locus of indeterminacy of $\phi$, as they are divisors in $\mbP^3$, which implies $\mathcal{F}_2$ has dicritical singularities. However, as the starting foliation induced by $Y\rightarrow T$ is a foliation with non-dicritical singularities, there is a divisor which is not invariant by the induced foliation at some stage of the foliated Sarkisov program and gets contracted to $\mbP^3$, which leads to a contradiction.
   
\end{example}

\begin{example}
    Let $C\subset Q\subset\mbP^3$ be a curve of genus $1$ and degree $4$, contained in a smooth quadric surface $Q$ with bi-degree $(2,2)$. Blowing up $\mbP^3$ along $C$ again gives a classical link of type 1, with the target $Y$ being a del Pezzo fibration over $T=\mbP^1$. If we consider the foliation $\mathcal{F}_1$ induced by this fibration, then as in the previous example, this Sarkisov link cannot be foliated.
\end{example}

\begin{remark}
   Note that, by \cite[Proposition $5.3$]{AD13}, every Fano foliation has strictly log canonical singularities, hence, it has dicritical singularities. Starting from a foliation with non-dicritical singularities, the construction of foliated Sarkisov links ensures we can never reach a foliated Mori fiber space whose foliation has dicritical singularities, in particular we can never reach a Picard rank one Fano variety.
\end{remark}

The above examples and remarks indicate that if a rational threefold admits a foliated Mori fiber space structure with a co-rank one F-dlt foliation, then there are no foliated Sarkisov links connecting it to $\mbP^3$. Consequently, it has fewer foliated Sarkisov links than classical Sarkisov links. It would be interesting to see if we can construct non-trivial subgroups of the Cremona group of $\mbP^3$ using birational automorphism groups of such F-dlt foliations on smooth rational threefolds.
\section{Rank One Foliated Mori Fiber Spaces}
In this section, our goal is to find a relationship between two rank one foliated Mori fiber spaces. However, as explained in the introduction, a Bertini-type theorem does not hold for rank one foliations on normal projective threefolds (this is not true even for algebraically integrable rank one foliations). Hence, the techniques used to prove the Sarkisov Program for co-rank one foliations in the previous sections fail, as they are heavily dependent on such theorems (cf. \cite[Lemmas $3.26$ and $3.27$]{CS21} and Lemma \ref{FL}).

Instead of proving the Sarkisov program for rank one foliations, we provide a more rigid relationship for rank one foliated Mori fiber spaces in Theorem \ref{MT2}. Before moving on to prove Theorem \ref{MT2}, we prove an important proposition concerning the structure of foliated Mori fiber spaces for rank one foliation.
\begin{proposition}\label{MFS1}
Let $X$ be a normal projective threefold and where $\mathcal{F}$ is a rank one foliation with canonical singularities on $X$. If $\phi:(X,\mathcal{F})\rightarrow S$ is a foliated Mori fiber space, then $\dim S=2$.
\end{proposition}
\begin{proof}

There are three possibilities for $\dim S$, it can be $0,1$ or $2$. If $\dim S=0$, then $\mathcal{F}$ is a Fano foliation on $X$, and $\mathcal{F}$ has at least strictly log canonical singularities by \cite[Proposition $5.3$]{AD13}, which is a contradiction to the fact that $\mathcal{F}$ has canonical singularities. Now suppose $\dim S=1$. Let $X_s$ be a general fiber. Since $X_s$ is invariant by $\mathcal{F}$, we can consider the restricted foliation $\mathcal{F}_s$ on $X_s$. $\phi$ being a foliated Mori fiber space implies $\mathcal{F}_s$ is Fano foliation on $X_s$, in particular it has a strictly log canonical singularity. Since $\mathcal{F}$ is rank one and $-K_{\mathcal{F}}$ is pseudo-effective, by \cite[Theorem $3.1$]{LMX23} $\mathcal{F}$ is algebraically integrable. Let $f:X\dashrightarrow Y$ be the rational map defining the foliation $\mathcal{F}$. Let $f':X'\rightarrow Y'$ be a property-$(*)$ modification of $\mathcal{F}$ with birational map $\pi:X'\rightarrow X$ (cf. \cite[Definition $3.8$]{ACSS}), which exists by \cite[Theorem $3.10$]{ACSS}. Since $\phi$ is tangent to the foliation $\mathcal{F}$, we have a morphism $g:Y'\rightarrow S$. Let $X'_s$ be a general fiber of the morphism $g\circ f':X'\rightarrow S$. The transformed foliation $\mathcal{F}'$ on $X'$ is induced by the morphism $f'$, in particular, $\mathcal{F}'|_{X'_s}:=\mathcal{F}'_{s}$ is non-dicritical as this is a foliation on the surface $X'_s$ which is defined by the morphism $f'|_{X'_s}$. For a general point $s$, the induced morphism $\pi_s:X'_{s}\rightarrow X_s$ is a birational morphism, with exceptional divisors coming from restriction of the exceptional divisors of $\pi$. As $\mathcal{F}'_s$ is a rank one non-dicritical foliation, there does not exist a birational $\pi':X''\rightarrow X'_s$ such that $\pi'$ extracts an exceptional divisor which is not $\pi'^{-1}\mathcal{F'}_s$-invariant. Hence, $\mathcal{F}_s$ being a foliation with a dicritical singularity imply $\pi_s$ extract a divisor which is $\mathcal{F}'_s$ non-invariant. Let $E$ be the reduced sum of all $\pi$-exceptional divisors. As $\mathcal{F}$ is a rank one foliation with canonical singularities (hence non-dicritical by \cite[Lemma $2.6$]{CS20}), each component of $E$ is $\mathcal{F}'$-invariant, which implies each component of $E|_{X'_s}$ is also $\mathcal{F}'_s$-invariant for a general point $s\in S$. However, all the $\pi_s$-exceptional divisors come from $E|_{X'_s}$, which should contain at least one $\pi_s$-exceptional divisor which is not $\mathcal{F}'_s$-invariant. This is our sought-after contradiction.
\end{proof}
\begin{remark}
    Proposition \ref{MFS1} confirms that in the above setup the foliation $\mathcal{F}$ is in fact induced by the morphism $\phi:X\rightarrow S$. However, the proof of the above proposition does not depend on dimension of $X$ being $3$. For any foliated Mori fiber space $\phi:(X,\mathcal{F})\rightarrow S$ with $\mathcal{F}$ a rank one foliation, if the general fiber of $\phi$ has dimension $\geqslant 2$, then we obtain a contradiction as in the proof. Hence, for any foliated Mori fiber space $\phi:(X,\mathcal{F})\rightarrow S$ with respect to a rank one foliation $\mathcal{F}$ with canonical singularities, $\mathcal{F}$ is induced by $\mathcal{\phi}$.  
\end{remark}
    
\begin{proof}[Proof of Theorem \ref{MT2}]
    Let the birational maps $p:Z\dashrightarrow X$ and $q:Z\dashrightarrow Y$ be two $(Z,\mathcal{F},\Delta)$-MMP. Since these maps are compositions of divisorial contractions and flips, their indeterminacy loci consist of surfaces and curves that are tangent to the foliation $\mathcal{F}$ by \cite{CS20}. By blowing up these loci, we obtain a morphism $W\rightarrow Z$ that resolves the indeterminacies of both $p$ and $q$. We thus have birational morphisms $p':W\rightarrow X$ and $q':W\rightarrow Y$, with respective exceptional loci $E_{p'}$ and $E_{q'}$. As the pair $(\mathcal{F},\Delta)$ on $Z$ has canonical singularities, the foliation $\mathcal{F}$ is non-dicritical by \cite[Corollary III.i.4]{MP13}. Hence, $E_{p'}$ and $E_{q'}$ are both invariant under the pullback foliation $\mathcal{F}_W$.\\

    By Proposition \ref{MFS1}, we only need to deal with the case $\dim S=\dim T=2$. Suppose that the union of the exceptional loci $E_{p'}\cup E_{q'}$ dominates either $S$ or $T$. Hence $E_{p'}\cup E_{q'}$ intersects the general fiber of $W\rightarrow S$. However, all components of $E_{p'}\cup E_{q'}$ are invariant, which implies if it intersects a curve tanget to the foliation, it contains the generic point of the curve. So, $E_{p'}\cup E_{q'}$ contains generic point of all curve in the fiber of $W\rightarrow S$, which is absurd.\\

    Therefore, the general fibers of the maps $W\rightarrow S$ and $W\rightarrow T$ must be disjoint from $E_{p'}\cup E_{q'}$. This implies that there exists a common open set $V\subset S$ (and $V \subset T$ via the induced map) such that for any $s\in V$, the fiber $W_s$ is mapped isomorphically to a fiber of $\psi:Y\rightarrow T$. Let $U=\phi^{-1}(V)$. Then the restriction $U\rightarrow V$ is the required common open foliated Mori fiber space.
    
\end{proof}

\begin{remark}
    If we consider our foliated pair $(\mathcal{F},\Delta)$ on $Z$ to be terminal (which corresponds to a smooth foliation in this context), it might be possible to establish the usual Sarkisov program as in Theorem \ref{MT} by using the same techniques as \cite{BM97}. The main ingredients, such as the existence and termination of the MMP and the ACC for log canonical thresholds, are already established by \cite{CS20} and \cite{YC22}. However, if $(\mathcal{F},\Delta)$ has canonical singularities, it is not clear how to define the ``Sarkisov triplet" (cf. \cite[Definition $1.4$]{BM97}). Hence, this technique fails for rank one foliations with canonical singularities.
\end{remark}

\bibliographystyle{alpha}
\bibliography{bibliography}

\end{document}